\documentclass[11pt,psamsfonts]{article}

\usepackage{longtable} 
\usepackage{pdflscape} 
\usepackage{lmodern}
\usepackage{url}
\usepackage{fancyvrb}

\usepackage{algorithm,algorithmic}
\usepackage{amsmath}
\usepackage{xtab}
\usepackage{amsthm}
\usepackage{amssymb}
\usepackage{amscd}
\usepackage{amsfonts}
\usepackage{amsbsy}
\usepackage{epsfig,afterpage}
\usepackage[dvips]{psfrag}
\usepackage{epsf}
\usepackage{psfrag}
\usepackage{color}     
\usepackage{graphicx}
\usepackage{epstopdf}
\usepackage[utf8]{inputenc}    
\usepackage{overpic}

\usepackage{epsfig}                      
\usepackage{epic,eepic}              
\usepackage{graphicx}                  
\usepackage{psfrag}                      
\usepackage[tight]{subfigure}      
\usepackage{pgfplots}                  
\usepackage{tikz}                           

\usepackage{amsbsy}         
\usepackage{color}               

\usepackage{tabularx}             
\usepackage{booktabs}          
\usepackage{multirow}           
\usepackage{longtable}          
\usepackage{array}

\newcommand{\R}{\ensuremath{\mathbb{R}}}
\newcommand{\C}{\ensuremath{\mathbb{C}}}
\newcommand{\N}{\ensuremath{\mathbb{N}}}
\newcommand{\Q}{\ensuremath{\mathbb{Q}}}
\newcommand{\Z}{\ensuremath{\mathbb{Z}}}
\newcommand{\p}{\partial}
\newcommand{\dis}{\displaystyle}
\newcommand{\vs}{\vspace{0,5cm}}
\newcommand{\sgn}{\mbox{\normalfont sgn}}
\newcommand{\T}{\mathbb{T}^2}

\newtheorem {theorem} {Theorem} 
\newtheorem {proposition} {Proposition}

\newtheorem {remark} {Remark}
\newtheorem {example} {Example}

\DeclareMathOperator{\Ext}{Ext}
\DeclareMathOperator{\End}{End}
\DeclareMathOperator{\Elm}{Elm}
\DeclareMathOperator{\Hom}{Hom}
\DeclareMathOperator{\Pic}{Pic}
\DeclareMathOperator{\Isom}{Isom}
\DeclareMathOperator{\Aut}{Aut}
\DeclareMathOperator{\length}{length}
\DeclareMathOperator{\coker}{coker}
\DeclareMathOperator{\Tot}{Tot}
\DeclareMathOperator{\Spec}{Spec}
\DeclareMathOperator{\vdeg}{\mathbf{deg}}
\DeclareMathOperator{\ad}{ad}
\DeclareMathOperator{\rk}{rk}
\DeclareMathOperator{\tr}{tr}
\DeclareMathOperator{\im}{im}
\DeclareMathOperator{\Diag}{Diag}
\DeclareMathOperator{\SHom}{\mathscr{H}{\!\mathit{om}}}
\DeclareMathOperator{\SEnd}{\mathscr{E}{\!\mathit{nd}}}
\DeclareMathOperator{\grad}{grad}
\DeclareMathOperator{\Ad}{Ad}
\DeclareMathOperator{\spann}{span}
\DeclareMathOperator{\ric}{Ric}
\newcolumntype{H}{>{\setbox0=\hbox\bgroup}c<{\egroup}@{}}

\allowdisplaybreaks
\begin{document}

\title{A numerical treatment to the problem of the quantity of Einstein metrics on flag manifolds}

\author{Lino Grama and Ricardo Miranda Martins}

%



\date{}

\maketitle


\begin{abstract}
In this paper we employ numerical methods to study the Einstein equation
\[
\ric(g)=\lambda\, g,
\]
where $\ric$ is the Ricci tensor and $\lambda$ is the Einstein constant, restricted to a class of full flag manifolds. These metrics describe the gravitational field of a vacuum with cosmological constant (vacuum is the case $\lambda=0$). In particular, we give estimates to the number of such metrics on the full flag manifolds $SU(n+1)/T^n$ for $n=4,5$, improving some classical estimatives. We also examine the isometric problem for these Einstein metrics. Our method can be applied for any fixed $n$.
\end{abstract}


\section{Introduction}\label{secao introducao}

In general relativity, the Lorentz metric $g$ is viewed as a gravitational potential. As such, it must be related, by a field equation, to the mass/energy distribution that generates the gravitational field. The field equations proposed by Einstein is 

\[\ric(g) -\frac12 g R + g\Lambda = T\]
where $\Lambda$ is the cosmological constant, $R$ is the scalar curvature and $T$ is the stress/energy tensor (cf. \cite{ga,besse}).


In this work we restrict ourselves to riemannian metrics in {\em homogeneous manifolds} and we are looking for solutions of Einstein equations that describe the gravitational field of a vacuum  ($T=0$) and $\Lambda=-\frac12\lambda$, that is, we need to find riemannian metrics $g$ satisfying 
\begin{equation}\label{einstein}
\ric(g)=\lambda\, g.
\end{equation}
Such metric $g$ is called {\em Einstein metric} and the constant $\lambda$ is called {\em Einstein constant}. By equation (\ref{einstein}), look for Einstein metrics is equivalent to solve a system of non-linear partial differential equations, and in general this is a very difficult task. 

If we consider Riemannian homogeneous manifolds (that is, manifolds that admits a transitive group of isometries) it is natural to ask by {\em invariant Einstein metrics}. In this case, the equation (\ref{einstein}) becomes a non-linear system of algebraic equations and find Einstein metrics is still a non-trivial problem.

In this paper we consider a class of homogeneous space called {\em full flag manifolds}, that is the homogeneous space $G/T$, where $G$ is a compact simple Lie group and $T$ is a maximal torus in $G$ and use numerial methods to solve the equation (\ref{einstein}). Invariant Einstein metrics on full flag manifolds was studied in several papers: Arvanitoyergos in \cite{arva1} classify the metrics on $SU(3)/T^2$ and gave an estimate for the number of invariant Eintein metric on the family $SU(n+1)/T^n$; Sakane in \cite{sakane} classify the invariant metrics on $SU(4)/T^3$ and more recently Arvanitoyergos-Crysikos-Sakane in \cite{ACS} classify the Einstein metrics on $G_2/T^2$. In these papers the autors use Gr\"obner Basis in order to solve the system of equations associated to the Einstein equation. We remark that the complexity of the system of equations (number of variables and number of equations)  depends on the number irreducible components of isotropy representation. In the examples above:  $SU(3)/T^2$ has $3$ components; $SU(4)/T^3$ has $6$ components and $G_2/T^2$ has $6$ components. 

The main result of this paper is to improve the estimates provided by Arvanitoyergos in \cite{arva1} for the full flag manifolds $SU(5)/T^4$ (with $10$ irreducible components), $SU(6)/T^5$ (with $15$ irreducible components) and provide a new insight to analyze the Einstein equations for invariant metrics on flag manifolds. We also examine the isometric problem for these Einstein metrics.

In particular we prove the following results (see Theorems \ref{t01} and \ref{F6}):\\

\noindent {\bf Theorem A:} The full flag manifold $SU(5)/T^4$ admits at least $396$ invariant Einstein metrics. If we consider invariant Einstein metrics up to {\em isometries} and {\em homoteties} we have at least $12$ class of such metrics.\\

\noindent {\bf Theorem B:}  The full flag manifold $SU(6)/T^5$ admit at least $3941$ invariant Einstein metrics. If we consider invariant Einstein metrics up to {\em isometries} and {\em homoteties} we have at least  $35$ class of such metrics.\\

We believe that Theorem A is sharp, that is, these 396 invariant Einstein metrics forms a complete set of metrics; conversely, we do not believe that Theorem B provide {\em all} invariant metrics on $SU(6)/T^5$.\\


The main tool to analyze the non-linear system of equations associated to the Einstein equation is numerical analysis and computational methods. All algorithms are described in details and can be easily adapted by the other flag manifolds. The computational cost of our methods are not studied in this paper, but we considered random initial conditions for the solve commands rather ran a predefined mesh, and this made the algorithms somewhat fast.

Recently, in \cite{china}, the authors have shown that there are {\it exactly} 29 $SU(4)$-invariant Einstein metrics on $SU(4)/T^3$, by solving a non-linear algebraic system with 6 variables and 6 equations using a numerical method (a single command on Maple). 
Unfortunately we could not reproduce their result. Our method allows to proof that there are {\it at least} 29 $SU(4)$-invariant Einstein metrics on $SU(4)/T^3$. 
However, it is well know that $SU(4)/T^3$ admits a finite number of Einstein metrics (see \cite{sakane}).
%
%

The rest of the paper is organized as follows. Section 2 presents a short introduction to the theory of flag manifolds and deduce the Einstein equations, giving the components of the Ricci tensor of invariant metrics on $SU(n+1)/T$. Section 3 introduces the K\"ahler-Einstein metric on flag manifolds, including the Borel-Hirzenbruck theorem on the existence of invariant complex structure on flag manifolds (
\cite{B-H}). Section 4 presents a invariant to detect when two invariant Einstein metric are {\em not} isometric. On Section 5 the main results are stated in details and proved. Section 6 presents details on the codes and algorithms we employed.

\section{Einstein equations on full flag manifolds}


Let $G^\mathbb{C}$ be complex simple non-compact Lie group and $G$ be a compact real form of $G^\mathbb{C}$. Let $B$ be a Borel subgroup of $G^\mathbb{C}$. We define the full flag manifold to be the homogeneous space $G^\mathbb{C}/B$. Let $T=B\cap G$ be a maximal torus in $G$. One can show that full flag manifold is diffeomorphic to homogeneous space $G/T$.  

We denote by $\mathfrak{g}$ and $\mathfrak{t}$ the Lie algebra of $G$ and $T$. Since $G/T$ is a reductive homogeneous space, the Lie algebra admits a split $\mathfrak{g}=\mathfrak{t}\oplus\mathfrak{m}$, where $\mathfrak{m}=\mathfrak{t}^\perp$ (with respect to the Cartan-Killing form of $\mathfrak{g}$).  The differential of the projection $\pi:G\rightarrow G/T$ induces an isomorphism between $\mathfrak{m}$ and $T_o(G/T)$, where $o=eT$ (trivial coset).

\begin{example}
Let $0\subset\ell\subset P\subset \mathbb{C}^3$ be a complete flag in $\mathbb{C}^3$, where $\ell$ is a line and $P$ is a 2-plane in $\mathbb{C}^3$. The set of  all complete flags in $\mathbb{C}^3$ is a 6-dimensional manifold diffeomorphic to $SU(3)/T^2$.  We can see this flag manifold as generalization of projective spaces and Grassmann manifolds (in fact, one can see projective spaces and Grassmann manifolds as {\em generalized flag manifolds}). More generally, the manifold of complete flags in $\mathbb{C}^{n+1}$ is diffeomorphic to the homogeneous space $SU(n+1)/T^n$. 
\end{example}

Let $g$ be an invariant metric on $G/T$. Since $G/T$ is a reductive homogeneous space, the metric $g$ is completely determined by its value at the origin $o=eT$ (trivial coset), see \cite{K-N}. In fact, there is a 1-1 correspondence between invariant metrics on $G/T$ and $\Ad (T)$-invariant scalar products on $\mathfrak{m}$. 

The isotropy representation of the full flag manifolds $G/T$ is completely reducible and the decomposition in non-equivalent sub-representations is given by $$\mathfrak{m}=\sum_{\alpha\in R^+}{\mathfrak{u}_\alpha}, $$where $\mathfrak{u}_\alpha=\spann_\mathbb{R}\{A_\alpha, S_\alpha\}, \alpha\in R^+$. It is well known that invariant tensors on $G/T$ are constant in each irreducible component of the isotropy representation and therefore are completely determined by its value at the origin $o$. In the case of invariant metrics we have the following description 
\begin{equation}
\{g=\langle \, , \,\rangle= \sum_{\alpha\in R^+ }{\lambda_\alpha \cdot Q|_{\mathfrak{u}_\alpha}; \lambda_{\alpha} \in \mathbb{R}^+}  \},
\end{equation}
where $Q=-B$, the negative of the Cartan-Killing form of $\mathfrak{g}$.

Similarly, the Ricci tensor $\ric(g)$ of a $G$-invariant metric $g$ on $G/T$ is an invariant tensor and its description as an $\Ad(T)$-invariant bilinear form on $\mathfrak{m}$ is given by
\begin{equation}
\ric(g)=\sum_{\alpha\in R^+}{r_\alpha\cdot Q|_{\mathfrak{u}_\alpha}},
\end{equation} 
where $r_\alpha, (\alpha\in R^+)$ are the components of the Ricci tensor on each irreducible component $\mathfrak{u}_\alpha$ of the isotropy representation.

The Ricci tensor for a invariant metric on full flag manifolds was computed by Sakane in \cite{sakane}.
\begin{proposition}\label{eq-einstein}
Let $G/T$ be a full flag manifold. For each $\alpha\in R^+$ the Ricci component $r_\alpha$ corresponding to the isotropy summand $\mathfrak{u}_\alpha$ is given by
\begin{equation}\label{eqs1}
r_\alpha=\frac{1}{2}\lambda_\alpha+\frac{1}{8}\sum_{\beta,\gamma\in R^+}{\frac{\lambda_\alpha}{\lambda_\beta \lambda_\gamma}   
\left[
\begin{array}{c}
\alpha\\
\beta\gamma
\end{array}
\right]
}-\frac{1}{4}\sum_{\beta,\gamma\in R^+}{\frac{\lambda_\gamma}{\lambda_\alpha\lambda_\beta	}   
\left[
\begin{array}{c}
\gamma\\
\alpha\beta
\end{array}
\right].
} 
\end{equation}
\end{proposition}
In the proposition \ref{eq-einstein} we use the notation $\left[
\begin{array}{c}
i\\
j \,  k
\end{array}
\right]$ introduced by Wang and Ziller \cite{W-Z} and defined as follows: let $G/H$ be a compact homogeneous space of a compact semissimple Lie group $G$ whose the isotropy representation $\mathfrak{m}$ decomposes into $k$ pairwise inequivalent irreducible $\Ad(H)$-submodules $\mathfrak{m}_i$ as $\mathfrak{m}=\mathfrak{m}_1\oplus \ldots \oplus \mathfrak{m}_k$. We choose a $Q$-orthonormal basis $\{e_p\}$ adapted to $\mathfrak{m}=\bigoplus_{i=1}^k\mathfrak{m}_i$. Let $A_{pq}^r=Q([e_p,e_q],e_r)$ so that $[e_p,e_q]_{\mathfrak{m}}=\sum A^r_{pq}e_r$, and set
\begin{equation}
\left[
\begin{array}{c}
i\\
j \,  k
\end{array}
\right]=\sum(A^r_{pq})^2=\sum( Q([e_p,e_q],e_r))^2,
\end{equation}
where the sum is taken over all indices $p,q,r$ with $e_p\in \mathfrak{m}_i$, $e_q\in \mathfrak{m}_j$ and $e_r\in \mathfrak{m}_k$. 

Now we consider the the flag manifolds $SU(n+1)/T$. Let $\mathfrak{su}(n+1)$ be the Lie algebra of $SU(n+1)$ and $\mathfrak{t}$ be the Lie algebra of $T$. A root system for the Cartan subalgebra $\mathfrak{t}^\mathbb{C}$ of the Lie algebra $\mathfrak{sl}(n+1,\mathbb{C})=(\mathfrak{su}(n+1))^\mathbb{C}$ is given by 
$$
R=\{ \alpha_{ij}=\varepsilon_i-\varepsilon_j| 1\leq i,j \leq n+1, i\neq j  \},
$$
and the set of positive roots 

$$
R^+=\{ \alpha_{ij}=\varepsilon_i-\varepsilon_j| 1\leq i<j \leq n+1, i\neq j  \}.
$$
The space of $SU(n+1)$-invariant metrics on $SU(n+1)/T$ is given by 
$$
\left\{ \sum_{i<j} \lambda_{ij}\cdot Q(\, ,\, )|_{\mathfrak{u}_{ij}} ;\,\, \lambda_{ij}\in \mathbb{R}^+ \right\}.
$$
The structure constants are given by
$$
\left\{
\begin{array}{l}
\left[
 \begin{array}{c}
 \alpha_{ij} \\
 \alpha_{ik}\, \alpha_{kj}
 \end{array}
 \right]=\frac{1}{n+1} \,\,\, (k\neq i,j) \\ \\
 0 \,\,\,\, \mbox{       otherwise}.
\end{array}
\right.
$$

\begin{proposition}{\cite{sakane}}
The components of the Ricci tensor of an invariant metric on $SU(n+1)/T$ are given by 
\begin{equation}\label{ricci-comp2}
r_{ij}=r_{\alpha_{ij}}=\frac{1}{2\lambda_{ij}}+\frac{1}{4(n+1)}+\sum_{k\neq i,j}\left( \frac{\lambda_{ij}}{\lambda_{ik}\lambda_{kj}}- \frac{\lambda_{ik}}{\lambda_{ij}\lambda_{kj}} -\frac{\lambda_{jk}}{\lambda_{ij}\lambda_{ik}}\right).
\end{equation}
\end{proposition}

An Eisntein metric is a metric with constant Ricci curvature, that is,  a solution of the system of equations 
\begin{equation}
r_{ij}=k,
\end{equation}
where $k$ is a constant (called {\em Einstein's constant}).

\begin{example}
Let us consider the full flag manifold $SU(3)/T^2$. Invariant Einstein metrics in this manifolds were studied by Arvanitoyergos in \cite{arva1}. In this case the Einstein equations are given by 
$$
\left\{
\begin{array}{ccccc}
r_{12}&=&\displaystyle\frac{1}{x_{12}} + \frac{1}{12}\left(  \frac{x_{12}}{x_{13}x_{23}} - \frac{x_{13}}{x_{12}x_{23}} - \frac{x_{23}}{x_{12}x_{13}} \right)&=& k\\ \\
r_{13}&=&\displaystyle\frac{1}{x_{13}} + \frac{1}{12}\left(  \frac{x_{13}}{x_{12}x_{23}} - \frac{x_{12}}{x_{13}x_{23}} - \frac{x_{23}}{x_{12}x_{13}} \right)&=& k\\ \\
r_{23}&=&\displaystyle\frac{1}{x_{23}} + \frac{1}{12}\left(  \frac{x_{23}}{x_{12}x_{13}} - \frac{x_{13}}{x_{12}x_{23}} - \frac{x_{12}}{x_{23}x_{13}} \right)&=& k.
\end{array}
\right.
$$
The solutions of the system are the $SU(3)$-invariant Einstein metrics(up to scalar multiple): $\{ x_{12}=1, x_{13}=1, x_{23}=1 \}$ (bi-invariant metric), $\{ x_{12}=2, x_{13}=1, x_{23}=1 \}$, $\{ x_{12}=1, x_{13}=2, x_{23}=1 \}$, $\{ x_{12}=1, x_{13}=1, x_{23}=2 \}$ (these last three metrics are K\"ahler-Einstein).
\end{example}

\section{K\"ahler-Einstein metric on flag manifolds}
Flag manifolds has a distinguish class of Einstein metrics called K\"ahler-Einstein metric. These metrics has interesting properties regarding to symplectic and riemannian geometry of flag manifolds. The description of such metrics can be done using the Lie theoretical properties of flag manifolds. We will restrict ourselves to the case $SU(n+1)/T^n$. A standard reference to this construct is \cite{besse}.

A invariant almost complex structure on $M=SU(n+1)/T^n$  if an endomorfism $J:T_oM\to T_oM$ such that $J^2=-Id$. Denote by $T^{1,0}M$ be the $i$-eigenspace of $J$. An invariant almost complex structure $J$ is called a invariant {\em complex} structure if $[T^{1,0}M,T^{1,0}M ] \subset T^{1,0}M $ (integrability condition). In this case, $(M,J)$ is a complex manifold.

The next classical result due to Borel-Hirzenbruck asserts about the existence of invariant complex structure on flag manifolds
\begin{theorem}[\cite{B-H}]\label{B-H}
There is a relation $1-1$ between invariant complex structure and orders of the root system.
\end{theorem}

Given an invariant metric $g$ on the complex manifold $(M,J)$ we define the fundamental $2$-form by $$ \omega (X,Y)=g(X,JY).$$ 

A {\em K\"ahler manifold} $(M,J,g)$ is a complex manifold with closed fundamental form $\omega$. In this case, $\omega$ is called K\"ahler form and $g$ is called K\"ahler metric. 

A natural question is when a K\"ahler metric is also an Einstein metric? The study of K\"ahler-Einstein metric is a classical subject in Differential Geometry. In the case of flag manifolds, there exist a (unique) canonical K\"ahler-Einstein associated to an invariant complex structure.

\begin{theorem}[\cite{besse}]

Let $G/K$ be a flag manifold. Then for each $G$-invariant complex structure $J$ on $G/K$, there exists a unique $G$-invariant K\"ahler-Einstein metric $g_J$ (up to scalar), given by
$$
J\leftrightarrow R^+_M \leftrightarrow g_J= \left\{  \lambda_\alpha= (\delta,\alpha): \delta=\frac{1}{2}\sum_{\beta\in R^+_M}{\beta}  \right\},
$$
where $R^+_M$ is the set of positive roots associated to $J$, by Theorem \ref{B-H}. 
\end{theorem}

\begin{remark}
In the case of $SU(n+1)/T^n$, we have $\frac{(n+1)!}{2}$ invariant K\"ahler-Einstein metrics (= number of invariant complex structures)
\end{remark}

\section{The isometric problem}
In this section we will review a method to provide a criteria to detect when two invariant Einstein metric are {\em not} isometric, cf. \cite{nik}, \cite{AC}. Applications of this method on isometric problem of generalized flag manifolds, see for instance \cite{AC}.

Let $SU(n+1)/T^n$ be the full flag manifold, with $$\mathfrak{m}=\sum_{\alpha\in R^+}{\mathfrak{u}_\alpha},$$ where $\dim \mathfrak{u}_\alpha=2$. Therefore, the dimension (real) of $SU(n+1)/T^n$ is $d=n(n+1)$. For any $SU(n+1)$-invariant metric we define a scale invariant $H_g=V^{1/d}S_g$, where $S_g$ is the scalar curvature of $g$ and $V=V_g/V_B$ is the quotient of the volumes $$V_g=\prod_{\alpha\in R^+}\lambda_\alpha^2$$of the metric $g$, and $V_B$ the volume of the normal metric induced by the negative of the Cartan-Killing form of $\mathfrak{su}(n+1)$. We normalize $V_B=1$, so $H_g=V_g^{1/d}S_g$. The scalar curvature of a $SU(n+1)$-invariant metric $g$ is 
$$S_g=2\sum_{\alpha\in R^+}{r_\alpha},$$  
where $r_\alpha$ are the components of the Ricci tensor computed in the equation (\ref{ricci-comp2}). The number $H_g$ is invariant under common scaling of the variables $\lambda_\alpha$.

If two invariant metrics are isometric then they have the same scale invariant, so if the scale invariant $H_{g_1}$ and $H_{g_2}$ are different, the invariant metric $g_1$ and $g_2$ can not be isometric. But if $H_{g_1}=H_{g_2}$ we can not immediately conclude if the metrics $g_1$ and $g_2$ are isometric or not. For such a case we have to look at the group of automorphisms of G and check if there is an automorphism which permutes the isotopy summands and takes one metric to another. This usually arises for the K\"ahler-Einstein metrics. Recall that the K\"ahler-Einstein metrics which correspond to equivalent invariant complex structures are isometric, cf. \cite{besse}.

\section{Main Results}

\begin{theorem}\label{t01}
The full flag manifold $SU(5)/T^4$ admit at least $396$ invariant Einstein metrics. If we consider invariant Einstein metrics up to {\em isometries} and {\em homoteties} we have at least $12$ class of such metrics.
\end{theorem}
\begin{proof}
We use the the Algorithm described in section \ref{solution-einstein} in order to obtain the solution the the system of  10 equations described in (\ref{ricci-comp2}). These solution are the Einstein metrics of $SU(5)/T^4$. We also compute the invariant scalar $H_g$ for each Einstein metric $g$. This invariant allow us to detect {\em non-isometric} Einstein metrics. See Table \ref{metric-f5}.
\end{proof}

\begin{remark}
In \cite{arva1}, Arvanitoyergos prove that $SU(n+1)/T^n$ admits at least $\frac{(n+1)!}{2}+n+1$ invariant Einstein metrics (the $\frac{(n+1)!}{2}$ Einstein metrics are the K\"ahler-Einstein metrics). Although we cannot guarantee that these metrics are all solutions of Einstein equations, we improved the estimate of the number of Einstein metrics on $SU(5)/T^4$. We believe that this number is optimal (see comments in Section \ref{solution-einstein}).
\end{remark}

In the very similar way we obtain an estimate to the number of Einstein metric for $SU(6)/T^5$.
\begin{theorem}\label{F6}
The full flag manifold $SU(6)/T^5$ admit at least $3941$ invariant Einstein metrics. If we consider invariant Einstein metrics up to {\em isometries} and {\em homoteties} we have at least  $35$ class of such metrics.
\end{theorem}

\begin{remark}
Theorem \ref{F6} also improve Arvanitoyergos's estimate to $SU(6)/T^5$. In this case, we believe that this number is not optimal, that is, there are metrics that our numerical method could not detect.
\end{remark}

\section{Algorithms}

In this section we present the Maple and Scilab algorithms we utilized on this paper. The full codes can be found on the website below: 
\begin{Verbatim}
http://www.ime.unicamp.br/~rmiranda/einstein-numerics/
\end{Verbatim}

We remark that these codes are not optimal, but they works for our purpose.

\subsection{Components of the Ricci tensor}

The folowing Maple worksheet generate the equations of the Ricci tensor. Given $m$, it generates a set of equations, denoted by \texttt{Difer3}, that is equivalent to system \eqref{eqs1} in dimension $m$. We remark that as the solutions of this system are invariant by scalar multiple, on line 44 of the algorithm below we set \texttt{z[1]:=1} to bypass the scaling.\\

\begin{Verbatim}[numbers=left]
restart; with(LinearAlgebra);
m := 5;
X := [];
for i to m do;
 for j to m do;
  if i < j then;
   X := [op(X), lambda[i, j]];
  end if;
 end do;
end do;
for i to m do;
 for j to m do;
  if i < j then;
   lambda[j, i] := lambda[i, j];
  end if;
 end do;
end do;
soma := proc (i, j) options operator, arrow; 
 (1/2)/lambda[i, j]+(1/4)*(sum(lambda[i, j]/(lambda[i, k]*lambda[k, j])-
  lambda[i, k]/(lambda[i, j]*lambda[k, j])-
  lambda[j, k]/(lambda[i, j]*lambda[i, k]), k = 1 .. m)
  -lambda[i, j]/(lambda[i, i]*lambda[i, j])-
  (-1)*lambda[i, i]/(lambda[i, j]*lambda[i, j])
  +lambda[j, i]/(lambda[i, i]*lambda[i, j])-
  lambda[i, j]/(lambda[i, j]*lambda[j, j])
  -(-1)*lambda[i, j]/(lambda[i, j]*lambda[j, j])-
  (-1)*lambda[j, j]/(lambda[i, j]*lambda[i, j]))/m;
end proc;
eq := [];
for i to m do;
 for j to m do;
  if i < j then;
   eq := [op(eq), soma(i, j)];
  end if;
 end do;
end do;
eqss := subs({seq(X[j] = z[j], j = 1 .. nops(x))}, eq);
difer := [];
for j to nops(eq)-1 do;
 difer := [op(difer), numer(simplify(eq[j]-eq[j+1]))];
end do;
Difer := subs({seq(X[j] = z[j], j = 1 .. nops(X))}, difer);
DiferSet := convert(Difer, set);
z[1] := 1;
Difer2 := subs({seq(z[j+1] = x[j], j = 1 .. nops(X)-1)}, DiferSet);
Difer3 := subs({seq(x[j] = x(j), j = 1 .. nops(X)-1)}, Difer2);
writeto("caso-m.txt");
save Difer3, "casoo-m.txt";
\end{Verbatim}

\subsection{Scilab codes for finding the metrics}\label{solution-einstein}

The algorithm below is basically a numerical root finding, considering a set of initial conditions on the $9$-dimensional cube $[0,10]^9$. The function $f$ on line 4 is obtained by the procedure above, for Maple. On line 9, the variable $K$ is the number of equations, and on line 11 the variable $L$ is the number of initial conditions. The algorithm will look for a solution near each initial condition and save to a file.

For $F_5$, we have a system of 9 equations and 9 variables. We started the algorithm with $10^6$ and $2\cdot 10^6$ initial conditions and have obtained, in both cases, the 396 metrics presented in Theorem \ref{t01}. We believe that this number is the exact number of the metrics.

For $F_6$ we keept the 100.000 initial conditions and found 1244 metrics. For 200.000 initial conditions, we found the 3941 of Theorem \ref{F6} - note that this space is $14$-dimensional. We dont know if this number if the exact number of the metrics.\\


\begin{Verbatim}[numbers=left]
n=getdate("s");
rand("seed",n);

function[f]=F(x)
f(1)=...;
...
f(K)=...;
endfunction;
ma=K;
resp = zeros(1,ma+1);
for n9=1:L;
 co =10*rand(1,ma);
 [vv1,vv2,vv3] = fsolve(co,F);
 resp = [resp; vv1 vv3];
end;
A=resp;

// control the quality of the solution
indices = find(A(:,ma+1)<>1);
 A(indices,:) = [];
A(:,ma+1) = [];

[badrows,c] = find(A<0.0001);
newA = A(setdiff(1:size(A,1),badrows),:);
B = round(100000*newA)/100000;
newB=unique(string(B),1);
save(sprintf('F8_%s.txt', string(rand())), newB)
\end{Verbatim}

\appendix

\section{Table of Einstein Metrics of $SU(5)/T^4$.}



\bigskip

{\tiny \hspace{-3cm}\begin{xtabular}{|ccccccccc|ccc|}
\hline $x_{1,3}$&$x_{1,4}$&$x_{1,5}$&$x_{2,3}$&$x_{2,4}$&$x_{2,5}$&$x_{3,4}$&$x_{3,5}$&$x_{4,5}$&{\rm Scalar \ curvature}&{\rm Volume}&{\rm Invariant $H_g$}\\
\hline 1,31177&1,31177&1,17466&1,31177&1,31177&1,17466&1&1,17466&1,17466&5,887017743&1,188797683&6,998473051\\
1,17466&1,31177&1,31177&1,17466&1,31177&1,31177&1,17466&1,17466&1&5,887017743&1,188797683&6,998473051\\
1,31177&1,17466&1,31177&1,31177&1,17466&1,31177&1,17466&1&1,17466&5,887017743&1,188797683&6,998473051\\
\hline 1&1&1&1,11672&1,11672&0,85131&0,85131&1,11672&1,11672&6,915252303&1,012034376&6,998473051\\
1&1&1&0,85131&1,11672&1,11672&1,11672&1,11672&0,85131&6,915252303&1,012034376&6,998473051\\
1,11672&0,85131&1,11672&1&1&1&1,11672&0,85131&1,11672&6,915252303&1,012034376&6,998473051\\
1,11672&1,11672&0,85131&1&1&1&0,85131&1,11672&1,11672&6,915252303&1,012034376&6,998473051\\
0,85131&1,11672&1,11672&1&1&1&1,11672&1,11672&0,85131&6,915252303&1,012034376&6,998473051\\
1&1&1&1,11672&0,85131&1,11672&1,11672&0,85131&1,11672&6,915252303&1,012034376&6,998473051\\
\hline 0,89548&0,76233&1&0,89548&1&0,76233&0,89548&0,89548&1&7,722400765&0,906256133&6,998473051\\
0,76233&0,89548&1&1&0,89548&0,76233&0,89548&1&0,89548&7,722400765&0,906256133&6,998473051\\
0,76233&1&0,89548&1&0,76233&0,89548&1&0,89548&0,89548&7,722400765&0,906256133&6,998473051\\
0,89548&1&0,76233&0,89548&0,76233&1&0,89548&0,89548&1&7,722400765&0,906256133&6,998473051\\
1&0,76233&0,89548&0,76233&1&0,89548&1&0,89548&0,89548&7,722400765&0,906256133&6,998473051\\
1&0,89548&0,76233&0,76233&0,89548&1&0,89548&1&0,89548&7,722400765&0,906256133&6,998473051\\
\hline 0,77656&1&1&1&0,87142&0,87142&1&1&0,87142&7,480600608&0,935590404&6,998778143\\
0,87142&0,87142&1&1&1&0,77656&0,87142&1&1&7,480600608&0,935590404&6,998778143\\
1&0,77656&1&0,87142&1&0,87142&1&0,87142&1&7,480600608&0,935590404&6,998778143\\
1&1&0,77656&0,87142&0,87142&1&0,87142&1&1&7,480600608&0,935590404&6,998778143\\
0,87142&1&0,87142&1&0,77656&1&1&0,87142&1&7,480600608&0,935590404&6,998778143\\
1&0,87142&0,87142&0,77656&1&1&1&1&0,87142&7,480600608&0,935590404&6,998778143\\
\hline 1,14755&1&1,14755&1,14755&1&1,14755&1,14755&0,89114&1,14755&6,518755057&1,073637233&6,998778143\\
1&1,14755&1,14755&1&1,14755&1,14755&1,14755&1,14755&0,89114&6,518755057&1,073637233&6,998778143\\
1,14755&1,14755&1&1,14755&1,14755&1&0,89114&1,14755&1,14755&6,518755057&1,073637233&6,998778143\\
\hline 1,28773&1,28773&1,28773&1,28773&1,28773&1,28773&1,12215&1,12215&1,12215&5,809142963&1,204786693&6,998778143\\
\hline 1&1&1&1&1&1&1&1&1&7&1&7\\
\hline 1,62727&1,02731&1,62727&1,62727&1,02731&1,62727&1,02731&1&1,02731&5,70282253&1,228186879&7,004131805\\
1,62727&1,62727&1,02731&1,62727&1,62727&1,02731&1&1,02731&1,02731&5,70282253&1,228186879&7,004131805\\
1,02731&1,62727&1,62727&1,02731&1,62727&1,62727&1,02731&1,02731&1&5,70282253&1,228186879&7,004131805\\
0,97341&1,584&1,584&1&1&1&1,584&1,584&0,97341&5,85858963&1,195532073&7,004131805\\
1&1&1&0,97341&1,584&1,584&1,584&1,584&0,97341&5,85858963&1,195532073&7,004131805\\
1&1&1&1,584&0,97341&1,584&1,584&0,97341&1,584&5,85858963&1,195532073&7,004131805\\
1&1&1&1,584&1,584&0,97341&0,97341&1,584&1,584&5,85858963&1,195532073&7,004131805\\
1,584&0,97341&1,584&1&1&1&1,584&0,97341&1,584&5,85858963&1,195532073&7,004131805\\
1,584&1,584&0,97341&1&1&1&0,97341&1,584&1,584&5,85858963&1,195532073&7,004131805\\
\hline 1&0,61453&0,63131&0,61453&1&0,63131&1&0,63131&0,63131&9,280013861&0,754754455&7,004131805\\
1&0,63131&0,61453&0,61453&0,63131&1&0,63131&1&0,63131&9,280013861&0,754754455&7,004131805\\
0,61453&0,63131&1&1&0,63131&0,61453&0,63131&1&0,63131&9,280013861&0,754754455&7,004131805\\
0,61453&1&0,63131&1&0,61453&0,63131&1&0,63131&0,63131&9,280013861&0,754754455&7,004131805\\
0,63131&0,61453&1&0,63131&1&0,61453&0,63131&0,63131&1&9,280013861&0,754754455&7,004131805\\
0,63131&1&0,61453&0,63131&0,61453&1&0,63131&0,63131&1&9,280013861&0,754754455&7,004131805\\
\hline 0,76697&1,07549&1,07549&1,43956&1,07971&1,07971&1,50792&1,50792&0,97129&6,218290494&1,126403696&7,004305396\\
1,07549&1,07549&0,76697&1,07971&1,07971&1,43956&0,97129&1,50792&1,50792&6,218290494&1,126403696&7,004305396\\
1,07971&1,43956&1,07971&1,07549&0,76697&1,07549&1,50792&0,97129&1,50792&6,218290494&1,126403696&7,004305396\\
1,07549&0,76697&1,07549&1,07971&1,43956&1,07971&1,50792&0,97129&1,50792&6,218290494&1,126403696&7,004305396\\
1,07971&1,07971&1,43956&1,07549&1,07549&0,76697&0,97129&1,50792&1,50792&6,218290494&1,126403696&7,004305396\\
1,43956&1,07971&1,07971&0,76697&1,07549&1,07549&1,50792&1,50792&0,97129&6,218290494&1,126403696&7,004305396\\
\hline 1,30384&1,40226&1,40226&1,87695&1,96609&1,96609&1,40776&1,40776&1,26641&4,769225176&1,468646402&7,004305396\\
1,40226&1,30384&1,40226&1,96609&1,87695&1,96609&1,40776&1,26641&1,40776&4,769225176&1,468646402&7,004305396\\
1,40226&1,40226&1,30384&1,96609&1,96609&1,87695&1,26641&1,40776&1,40776&4,769225176&1,468646402&7,004305396\\
1,87695&1,96609&1,96609&1,30384&1,40226&1,40226&1,40776&1,40776&1,26641&4,769225176&1,468646402&7,004305396\\
1,96609&1,87695&1,96609&1,40226&1,30384&1,40226&1,40776&1,26641&1,40776&4,769225176&1,468646402&7,004305396\\
1,96609&1,96609&1,87695&1,40226&1,40226&1,30384&1,26641&1,40776&1,40776&4,769225176&1,468646402&7,004305396\\
\hline 1&0,92618&1,33329&0,89959&0,99609&1,3966&0,99609&1,3966&0,71035&6,713936781&1,043248637&7,004305396\\
0,99609&1,3966&0,89959&0,92618&1,33329&1&0,71035&0,99609&1,3966&6,713936781&1,043248637&7,004305396\\
0,89959&0,99609&1,3966&1&0,92618&1,33329&0,99609&1,3966&0,71035&6,713936781&1,043248637&7,004305396\\
0,89959&1,3966&0,99609&1&1,33329&0,92618&1,3966&0,99609&0,71035&6,713936781&1,043248637&7,004305396\\
1&1,33329&0,92618&0,89959&1,3966&0,99609&1,3966&0,99609&0,71035&6,713936781&1,043248637&7,004305396\\
0,92618&1,33329&1&0,99609&1,3966&0,89959&0,71035&0,99609&1,3966&6,713936781&1,043248637&7,004305396\\
0,99609&0,89959&1,3966&0,92618&1&1,33329&0,99609&0,71035&1,3966&6,713936781&1,043248637&7,004305396\\
1,33329&0,92618&1&1,3966&0,99609&0,89959&0,71035&1,3966&0,99609&6,713936781&1,043248637&7,004305396\\
1,33329&1&0,92618&1,3966&0,89959&0,99609&1,3966&0,71035&0,99609&6,713936781&1,043248637&7,004305396\\
1,3966&0,89959&0,99609&1,33329&1&0,92618&1,3966&0,71035&0,99609&6,713936781&1,043248637&7,004305396\\
1,3966&0,99609&0,89959&1,33329&0,92618&1&0,71035&1,3966&0,99609&6,713936781&1,043248637&7,004305396\\
0,92618&1&1,33329&0,99609&0,89959&1,3966&0,99609&0,71035&1,3966&6,713936781&1,043248637&7,004305396\\
\hline 0,71313&0,92981&1&1,40208&1,00392&0,90312&1,33852&1,40208&1,00392&6,687709023&1,047340034&7,004305396\\
0,71313&1&0,92981&1,40208&0,90312&1,00392&1,40208&1,33852&1,00392&6,687709023&1,047340034&7,004305396\\
0,92981&0,71313&1&1,00392&1,40208&0,90312&1,33852&1,00392&1,40208&6,687709023&1,047340034&7,004305396\\
0,92981&1&0,71313&1,00392&0,90312&1,40208&1,00392&1,33852&1,40208&6,687709023&1,047340034&7,004305396\\
1&0,71313&0,92981&0,90312&1,40208&1,00392&1,40208&1,00392&1,33852&6,687709023&1,047340034&7,004305396\\
1&0,92981&0,71313&0,90312&1,00392&1,40208&1,00392&1,40208&1,33852&6,687709023&1,047340034&7,004305396\\
1,00392&0,90312&1,40208&0,92981&1&0,71313&1,00392&1,33852&1,40208&6,687709023&1,047340034&7,004305396\\
1,00392&1,40208&0,90312&0,92981&0,71313&1&1,33852&1,00392&1,40208&6,687709023&1,047340034&7,004305396\\
1,40208&0,90312&1,00392&0,71313&1&0,92981&1,40208&1,33852&1,00392&6,687709023&1,047340034&7,004305396\\
1,40208&1,00392&0,90312&0,71313&0,92981&1&1,33852&1,40208&1,00392&6,687709023&1,047340034&7,004305396\\
0,90312&1,00392&1,40208&1&0,92981&0,71313&1,00392&1,40208&1,33852&6,687709023&1,047340034&7,004305396\\
0,90312&1,40208&1,00392&1&0,71313&0,92981&1,40208&1,00392&1,33852&6,687709023&1,047340034&7,004305396 \\ 
\hline 0,69466&0,75003&0,75003&0,53278&1,04749&1,04749&0,74709&0,74709&0,67471&8,951607873&0,782463385&7,004305396\\
0,75003&0,69466&0,75003&1,04749&0,53278&1,04749&0,74709&0,67471&0,74709&8,951607873&0,782463385&7,004305396\\
0,75003&0,75003&0,69466&1,04749&1,04749&0,53278&0,67471&0,74709&0,74709&8,951607873&0,782463385&7,004305396\\
1,04749&0,53278&1,04749&0,75003&0,69466&0,75003&0,74709&0,67471&0,74709&8,951607873&0,782463385&7,004305396\\
1,04749&1,04749&0,53278&0,75003&0,75003&0,69466&0,67471&0,74709&0,74709&8,951607873&0,782463385&7,004305396\\
0,53278&1,04749&1,04749&0,69466&0,75003&0,75003&0,74709&0,74709&0,67471&8,951607873&0,782463385&7,004305396\\
\hline 1,10727&1,11162&1,55249&1,10727&1,11162&1,55249&1,02956&0,78964&1,48211&6,039781261&1,15969521&7,004305396\\
1,11162&1,10727&1,55249&1,11162&1,10727&1,55249&1,02956&1,48211&0,78964&6,039781261&1,15969521&7,004305396\\
1,11162&1,55249&1,10727&1,11162&1,55249&1,10727&1,48211&1,02956&0,78964&6,039781261&1,15969521&7,004305396\\
1,10727&1,55249&1,11162&1,10727&1,55249&1,11162&0,78964&1,02956&1,48211&6,039781261&1,15969521&7,004305396\\
1,55249&1,10727&1,11162&1,55249&1,10727&1,11162&0,78964&1,48211&1,02956&6,039781261&1,15969521&7,004305396\\
1,55249&1,11162&1,10727&1,55249&1,11162&1,10727&1,48211&0,78964&1,02956&6,039781261&1,15969521&7,004305396\\
\hline 0,50863&0,95466&1&0,71322&0,71602&0,64413&0,66316&0,71322&0,71602&9,376722977&0,746988624&7,004305396\\
0,50863&1&0,95466&0,71322&0,64413&0,71602&0,71322&0,66316&0,71602&9,376722977&0,746988624&7,004305396\\
0,64413&0,71322&0,71602&1&0,50863&0,95466&0,71322&0,71602&0,66316&9,376722977&0,746988624&7,004305396\\
0,64413&0,71602&0,71322&1&0,95466&0,50863&0,71602&0,71322&0,66316&9,376722977&0,746988624&7,004305396\\
0,71322&0,64413&0,71602&0,50863&1&0,95466&0,71322&0,66316&0,71602&9,376722977&0,746988624&7,004305396\\
0,71322&0,71602&0,64413&0,50863&0,95466&1&0,66316&0,71322&0,71602&9,376722977&0,746988624&7,004305396\\
0,71602&0,64413&0,71322&0,95466&1&0,50863&0,71602&0,66316&0,71322&9,376722977&0,746988624&7,004305396\\
0,71602&0,71322&0,64413&0,95466&0,50863&1&0,66316&0,71602&0,71322&9,376722977&0,746988624&7,004305396\\
0,95466&0,50863&1&0,71602&0,71322&0,64413&0,66316&0,71602&0,71322&9,376722977&0,746988624&7,004305396\\
0,95466&1&0,50863&0,71602&0,64413&0,71322&0,71602&0,66316&0,71322&9,376722977&0,746988624&7,004305396\\
1&0,50863&0,95466&0,64413&0,71322&0,71602&0,71322&0,71602&0,66316&9,376722977&0,746988624&7,004305396\\
1&0,95466&0,50863&0,64413&0,71602&0,71322&0,71602&0,71322&0,66316&9,376722977&0,746988624&7,004305396\\
\hline 1&1,01017&1,42762&1&1,01017&1,42762&1,01017&1,42762&0,73658&6,470618953&1,082490564&7,004383959\\
1&1,42762&1,01017&1&1,42762&1,01017&1,42762&1,01017&0,73658&6,470618953&1,082490564&7,004383959\\
1,01017&1&1,42762&1,01017&1&1,42762&1,01017&0,73658&1,42762&6,470618953&1,082490564&7,004383959\\
1,01017&1,42762&1&1,01017&1,42762&1&0,73658&1,01017&1,42762&6,470618953&1,082490564&7,004383959\\
1,42762&1&1,01017&1,42762&1&1,01017&1,42762&0,73658&1,01017&6,470618953&1,082490564&7,004383959\\
1,42762&1,01017&1&1,42762&1,01017&1&0,73658&1,42762&1,01017&6,470618953&1,082490564&7,004383959\\
\hline 1,93817&1,93817&1,93817&1,37143&1,37143&1,37143&1,35762&1,35762&1,35762&4,766140305&1,469613463&7,004383959\\
1,37143&1,37143&1,37143&1,93817&1,93817&1,93817&1,35762&1,35762&1,35762&4,766140305&1,469613463&7,004383959\\
\hline 1&0,51595&1&0,70047&0,70759&0,70047&0,70759&0,70047&0,70759&9,237571548&0,758249495&7,004383959\\
1&1&0,51595&0,70047&0,70047&0,70759&0,70047&0,70759&0,70759&9,237571548&0,758249495&7,004383959\\
0,51595&1&1&0,70759&0,70047&0,70047&0,70759&0,70759&0,70047&9,237571548&0,758249495&7,004383959\\
0,70047&0,70759&0,70047&1&0,51595&1&0,70759&0,70047&0,70759&9,237571548&0,758249495&7,004383959\\
0,70759&0,70047&0,70047&0,51595&1&1&0,70759&0,70759&0,70047&9,237571548&0,758249495&7,004383959\\
0,70047&0,70047&0,70759&1&1&0,51595&0,70047&0,70759&0,70759&9,237571548&0,758249495&7,004383959\\
0,98993&0,98993&1,41324&1&1&0,72916&0,98993&1,41324&1,41324&6,536443914&1,071589392&7,004383959\\
0,98993&1,41324&0,98993&1&0,72916&1&1,41324&0,98993&1,41324&6,536443914&1,071589392&7,004383959\\
1,41324&0,98993&0,98993&0,72916&1&1&1,41324&1,41324&0,98993&6,536443914&1,071589392&7,004383959\\
0,72916&1&1&1,41324&0,98993&0,98993&1,41324&1,41324&0,98993&6,536443914&1,071589392&7,004383959\\
1&0,72916&1&0,98993&1,41324&0,98993&1,41324&0,98993&1,41324&6,536443914&1,071589392&7,004383959\\
1&1&0,72916&0,98993&0,98993&1,41324&0,98993&1,41324&1,41324&6,536443914&1,071589392&7,004383959\\
\hline 1,55678&1,10258&1,55678&1,55678&1,10258&1,55678&0,92514&1,41718&0,92514&5,647817174&1,240957513&7,008701155\\
1,55678&1,55678&1,10258&1,55678&1,55678&1,10258&1,41718&0,92514&0,92514&5,647817174&1,240957513&7,008701155\\
1,10258&1,55678&1,55678&1,10258&1,55678&1,55678&0,92514&0,92514&1,41718&5,647817174&1,240957513&7,008701155\\
\hline 0,59426&0,91033&1&0,70824&1&0,64235&0,59426&0,70824&1&8,792433536&0,797128705&7,008701155\\
0,59426&1&0,91033&0,70824&0,64235&1&0,70824&0,59426&1&8,792433536&0,797128705&7,008701155\\
0,70824&0,64235&1&0,59426&1&0,91033&0,70824&0,59426&1&8,792433536&0,797128705&7,008701155\\
0,70824&1&0,64235&0,59426&0,91033&1&0,59426&0,70824&1&8,792433536&0,797128705&7,008701155\\
0,64235&0,70824&1&1&0,59426&0,91033&0,70824&1&0,59426&8,792433536&0,797128705&7,008701155\\
0,64235&1&0,70824&1&0,91033&0,59426&1&0,70824&0,59426&8,792433536&0,797128705&7,008701155\\
0,91033&0,59426&1&1&0,70824&0,64235&0,59426&1&0,70824&8,792433536&0,797128705&7,008701155\\
0,91033&1&0,59426&1&0,64235&0,70824&1&0,59426&0,70824&8,792433536&0,797128705&7,008701155\\
1&0,59426&0,91033&0,64235&0,70824&1&0,70824&1&0,59426&8,792433536&0,797128705&7,008701155\\
1&0,64235&0,70824&0,91033&1&0,59426&1&0,59426&0,70824&8,792433536&0,797128705&7,008701155\\
1&0,70824&0,64235&0,91033&0,59426&1&0,59426&1&0,70824&8,792433536&0,797128705&7,008701155\\
1&0,91033&0,59426&0,64235&1&0,70824&1&0,70824&0,59426&8,792433536&0,797128705&7,008701155\\
\hline 0,83907&0,83907&1&1,41195&1,41195&0,90697&1,28534&1,41195&1,41195&6,227146486&1,125507674&7,008701155\\
0,83907&1&0,83907&1,41195&0,90697&1,41195&1,41195&1,28534&1,41195&6,227146486&1,125507674&7,008701155\\
0,90697&1,41195&1,41195&1&0,83907&0,83907&1,41195&1,41195&1,28534&6,227146486&1,125507674&7,008701155\\
1,41195&1,41195&0,90697&0,83907&0,83907&1&1,28534&1,41195&1,41195&6,227146486&1,125507674&7,008701155\\
1&0,83907&0,83907&0,90697&1,41195&1,41195&1,41195&1,41195&1,28534&6,227146486&1,125507674&7,008701155\\
1,41195&0,90697&1,41195&0,83907&1&0,83907&1,41195&1,28534&1,41195&6,227146486&1,125507674&7,008701155\\
\hline 1,53186&1,68276&1,68276&1&1,1918&1,1918&1,68276&1,68276&1,08092&5,225005252&1,341376863&7,008701155\\
1,68276&1,53186&1,68276&1,1918&1&1,1918&1,68276&1,08092&1,68276&5,225005252&1,341376863&7,008701155\\
1,68276&1,68276&1,53186&1,1918&1,1918&1&1,08092&1,68276&1,68276&5,225005252&1,341376863&7,008701155\\
1&1,1918&1,1918&1,53186&1,68276&1,68276&1,68276&1,68276&1,08092&5,225005252&1,341376863&7,008701155\\
1,1918&1&1,1918&1,68276&1,53186&1,68276&1,68276&1,08092&1,68276&5,225005252&1,341376863&7,008701155\\
1,1918&1,1918&1&1,68276&1,68276&1,53186&1,08092&1,68276&1,68276&5,225005252&1,341376863&7,008701155\\
\hline 0,6528&1,0985&1,0985&0,6528&1,0985&1,0985&0,778&0,778&0,70562&8,004025824&0,875646994&7,008701155\\
1,0985&0,6528&1,0985&1,0985&0,6528&1,0985&0,778&0,70562&0,778&8,004025824&0,875646994&7,008701155\\
1,0985&1,0985&0,6528&1,0985&1,0985&0,6528&0,70562&0,778&0,778&8,004025824&0,875646994&7,008701155\\
\hline 1,9045&2,54164&1,93057&1,21247&1,9045&1,29494&1&1,29494&1,93057&4,586782706&1,529050183&7,013420937\\
2,54164&1,9045&1,93057&1,9045&1,21247&1,29494&1&1,93057&1,29494&4,586782706&1,529050183&7,013420937\\
2,54164&1,93057&1,9045&1,9045&1,29494&1,21247&1,93057&1&1,29494&4,586782706&1,529050183&7,013420937\\
1,21247&1,29494&1,9045&1,9045&1,93057&2,54164&1,29494&1&1,93057&4,586782706&1,529050183&7,013420937\\
1,21247&1,9045&1,29494&1,9045&2,54164&1,93057&1&1,29494&1,93057&4,586782706&1,529050183&7,013420937\\
1,29494&1,21247&1,9045&1,93057&1,9045&2,54164&1,29494&1,93057&1&4,586782706&1,529050183&7,013420937\\
1,29494&1,9045&1,21247&1,93057&2,54164&1,9045&1,93057&1,29494&1&4,586782706&1,529050183&7,013420937\\
1,9045&1,21247&1,29494&2,54164&1,9045&1,93057&1&1,93057&1,29494&4,586782706&1,529050183&7,013420937\\
1,9045&1,29494&1,21247&2,54164&1,93057&1,9045&1,93057&1&1,29494&4,586782706&1,529050183&7,013420937\\
1,9045&1,93057&2,54164&1,21247&1,29494&1,9045&1,29494&1&1,93057&4,586782706&1,529050183&7,013420937\\
1,93057&1,9045&2,54164&1,29494&1,21247&1,9045&1,29494&1,93057&1&4,586782706&1,529050183&7,013420937\\
1,93057&2,54164&1,9045&1,29494&1,9045&1,21247&1,93057&1,29494&1&4,586782706&1,529050183&7,013420937\\
\hline 1,57076&1,06802&0,82476&0,82476&1,06802&1,57076&1,59226&2,09625&1,59226&5,561339899&1,261102731&7,013420937\\
1,06802&0,82476&1,57076&1,06802&1,57076&0,82476&1,59226&1,59226&2,09625&5,561339899&1,261102731&7,013420937\\
1,06802&1,57076&0,82476&1,06802&0,82476&1,57076&1,59226&1,59226&2,09625&5,561339899&1,261102731&7,013420937\\
0,82476&1,06802&1,57076&1,57076&1,06802&0,82476&1,59226&2,09625&1,59226&5,561339899&1,261102731&7,013420937\\
0,82476&1,57076&1,06802&1,57076&0,82476&1,06802&2,09625&1,59226&1,59226&5,561339899&1,261102731&7,013420937\\
1,57076&0,82476&1,06802&0,82476&1,57076&1,06802&2,09625&1,59226&1,59226&5,561339899&1,261102731&7,013420937\\
\hline 0,52507&0,67994&0,63663&1,33454&1,01369&0,52507&1,01369&1&0,67994&8,735533715&0,802861183&7,013420937\\
0,52507&1,33454&1,01369&0,63663&0,52507&0,67994&1&0,67994&1,01369&8,735533715&0,802861183&7,013420937\\
0,63663&0,67994&0,52507&0,52507&1,01369&1,33454&0,67994&1&1,01369&8,735533715&0,802861183&7,013420937\\
0,67994&0,63663&0,52507&1,01369&0,52507&1,33454&0,67994&1,01369&1&8,735533715&0,802861183&7,013420937\\
1,01369&1,33454&0,52507&0,67994&0,52507&0,63663&1,01369&0,67994&1&8,735533715&0,802861183&7,013420937\\
0,52507&0,63663&0,67994&1,33454&0,52507&1,01369&1&1,01369&0,67994&8,735533715&0,802861183&7,013420937\\
0,52507&1,01369&1,33454&0,63663&0,67994&0,52507&0,67994&1&1,01369&8,735533715&0,802861183&7,013420937\\
0,63663&0,52507&0,67994&0,52507&1,33454&1,01369&1&0,67994&1,01369&8,735533715&0,802861183&7,013420937\\
0,67994&0,52507&0,63663&1,01369&1,33454&0,52507&1,01369&0,67994&1&8,735533715&0,802861183&7,013420937\\
1,01369&0,52507&1,33454&0,67994&0,63663&0,52507&0,67994&1,01369&1&8,735533715&0,802861183&7,013420937\\
1,33454&0,52507&1,01369&0,52507&0,63663&0,67994&1&1,01369&0,67994&8,735533715&0,802861183&7,013420937\\
1,33454&1,01369&0,52507&0,52507&0,67994&0,63663&1,01369&1&0,67994&8,735533715&0,802861183&7,013420937\\
\hline 0,77224&0,93631&1,47072&1,49086&1&1,49086&1,47072&1,96275&0,77224&5,939605625&1,180788992&7,013420937\\
0,77224&1,47072&0,93631&1,49086&1,49086&1&1,96275&1,47072&0,77224&5,939605625&1,180788992&7,013420937\\
1&1,49086&1,49086&0,93631&0,77224&1,47072&1,47072&0,77224&1,96275&5,939605625&1,180788992&7,013420937\\
1,47072&0,77224&0,93631&1,49086&1,49086&1&1,96275&0,77224&1,47072&5,939605625&1,180788992&7,013420937\\
1,47072&0,93631&0,77224&1,49086&1&1,49086&0,77224&1,96275&1,47072&5,939605625&1,180788992&7,013420937\\
1,49086&1,49086&1&0,77224&1,47072&0,93631&1,96275&1,47072&0,77224&5,939605625&1,180788992&7,013420937\\
0,93631&0,77224&1,47072&1&1,49086&1,49086&1,47072&0,77224&1,96275&5,939605625&1,180788992&7,013420937\\
0,93631&1,47072&0,77224&1&1,49086&1,49086&0,77224&1,47072&1,96275&5,939605625&1,180788992&7,013420937\\
1&1,49086&1,49086&0,93631&1,47072&0,77224&0,77224&1,47072&1,96275&5,939605625&1,180788992&7,013420937\\
1,49086&1&1,49086&0,77224&0,93631&1,47072&1,47072&1,96275&0,77224&5,939605625&1,180788992&7,013420937\\
1,49086&1&1,49086&1,47072&0,93631&0,77224&0,77224&1,96275&1,47072&5,939605625&1,180788992&7,013420937\\
1,49086&1,49086&1&1,47072&0,77224&0,93631&1,96275&0,77224&1,47072&5,939605625&1,180788992&7,013420937\\
\hline 0,67076&0,67076&1&0,51798&0,9865&1,31652&0,62804&0,9865&0,51798&8,855090039&0,792021414&7,013420937\\
0,9865&0,51798&1,31652&0,67076&0,67076&1&0,62804&0,51798&0,9865&8,855090039&0,792021414&7,013420937\\
0,9865&1,31652&0,51798&0,67076&1&0,67076&0,51798&0,62804&0,9865&8,855090039&0,792021414&7,013420937\\
1&0,67076&0,67076&1,31652&0,9865&0,51798&0,51798&0,9865&0,62804&8,855090039&0,792021414&7,013420937\\
0,51798&0,9865&1,31652&0,67076&0,67076&1&0,62804&0,9865&0,51798&8,855090039&0,792021414&7,013420937\\
0,67076&0,67076&1&0,9865&0,51798&1,31652&0,62804&0,51798&0,9865&8,855090039&0,792021414&7,013420937\\
0,67076&1&0,67076&0,51798&1,31652&0,9865&0,9865&0,62804&0,51798&8,855090039&0,792021414&7,013420937\\
0,67076&1&0,67076&0,9865&1,31652&0,51798&0,51798&0,62804&0,9865&8,855090039&0,792021414&7,013420937\\
1&0,67076&0,67076&1,31652&0,51798&0,9865&0,9865&0,51798&0,62804&8,855090039&0,792021414&7,013420937\\
1,31652&0,9865&0,51798&1&0,67076&0,67076&0,51798&0,9865&0,62804&8,855090039&0,792021414&7,013420937\\
0,51798&1,31652&0,9865&0,67076&1&0,67076&0,9865&0,62804&0,51798&8,855090039&0,792021414&7,013420937\\
1,31652&0,51798&0,9865&1&0,67076&0,67076&0,9865&0,51798&0,62804&8,855090039&0,792021414&7,013420937\\
\hline 0,74932&0,39345&0,75958&0,39345&0,74932&0,75958&0,47704&0,50949&0,50949&11,65792384&0,601601197&7,013420937\\
0,75958&0,74932&0,39345&0,75958&0,39345&0,74932&0,50949&0,50949&0,47704&11,65792384&0,601601197&7,013420937\\
0,39345&0,74932&0,75958&0,74932&0,39345&0,75958&0,47704&0,50949&0,50949&11,65792384&0,601601197&7,013420937\\
0,39345&0,75958&0,74932&0,74932&0,75958&0,39345&0,50949&0,47704&0,50949&11,65792384&0,601601197&7,013420937\\
0,74932&0,75958&0,39345&0,39345&0,75958&0,74932&0,50949&0,47704&0,50949&11,65792384&0,601601197&7,013420937\\
0,75958&0,39345&0,74932&0,75958&0,74932&0,39345&0,50949&0,50949&0,47704&11,65792384&0,601601197&7,013420937\\
\hline 0,66667&1&1&0,66667&1&1&0,66667&0,66667&1&8,2499835&0,850284701&7,014834753\\
1&0,66667&1&1&0,66667&1&0,66667&1&0,66667&8,2499835&0,850284701&7,014834753\\
1&1&0,66667&1&1&0,66667&1&0,66667&0,66667&8,2499835&0,850284701&7,014834753\\
\hline 1&1&1&1,5&1,5&1,5&1,5&1,5&1,5&5,5&1,275424501&7,014834753\\
1,5&1,5&1,5&1&1&1&1,5&1,5&1,5&5,5&1,275424501&7,014834753\\
\hline 0,87699&1,42308&1,98274&0,87699&1,42308&1,98274&0,75527&1,36724&0,70089&6,048359226&1,160533364&7,019322677\\
0,87699&1,98274&1,42308&0,87699&1,98274&1,42308&1,36724&0,75527&0,70089&6,048359226&1,160533364&7,019322677\\
1,42308&0,87699&1,98274&1,42308&0,87699&1,98274&0,75527&0,70089&1,36724&6,048359226&1,160533364&7,019322677\\
1,42308&1,98274&0,87699&1,42308&1,98274&0,87699&0,70089&0,75527&1,36724&6,048359226&1,160533364&7,019322677\\
1,98274&0,87699&1,42308&1,98274&0,87699&1,42308&1,36724&0,70089&0,75527&6,048359226&1,160533364&7,019322677\\
1,98274&1,42308&0,87699&1,98274&1,42308&0,87699&0,70089&1,36724&0,75527&6,048359226&1,160533364&7,019322677\\
\hline 1,16116&1,16116&1,81028&1,88421&1,88421&0,928&1,32404&2,62521&2,62521&4,568131946&1,536584924&7,019322677\\
1,16116&1,81028&1,16116&1,88421&0,928&1,88421&2,62521&1,32404&2,62521&4,568131946&1,536584924&7,019322677\\
1,88421&1,88421&0,928&1,16116&1,16116&1,81028&1,32404&2,62521&2,62521&4,568131946&1,536584924&7,019322677\\
0,928&1,88421&1,88421&1,81028&1,16116&1,16116&2,62521&2,62521&1,32404&4,568131946&1,536584924&7,019322677\\
1,81028&1,16116&1,16116&0,928&1,88421&1,88421&2,62521&2,62521&1,32404&4,568131946&1,536584924&7,019322677\\
1,88421&0,928&1,88421&1,16116&1,81028&1,16116&2,62521&1,32404&2,62521&4,568131946&1,536584924&7,019322677\\
\hline 1,45017&1,45017&0,51263&0,64143&0,64143&0,5524&0,7314&1,04084&1,04084&8,269581109&0,848812362&7,019322677\\
0,5524&0,64143&0,64143&0,51263&1,45017&1,45017&1,04084&1,04084&0,7314&8,269581109&0,848812362&7,019322677\\
0,64143&0,5524&0,64143&1,45017&0,51263&1,45017&1,04084&0,7314&1,04084&8,269581109&0,848812362&7,019322677\\
0,64143&0,64143&0,5524&1,45017&1,45017&0,51263&0,7314&1,04084&1,04084&8,269581109&0,848812362&7,019322677\\
1,45017&0,51263&1,45017&0,64143&0,5524&0,64143&1,04084&0,7314&1,04084&8,269581109&0,848812362&7,019322677\\
0,51263&1,45017&1,45017&0,5524&0,64143&0,64143&1,04084&1,04084&0,7314&8,269581109&0,848812362&7,019322677\\
1&0,8612&1,55902&1,14027&1,62269&2,26085&1,62269&2,26085&0,7992&5,304346149&1,323315349&7,019322677\\
1&1,55902&0,8612&1,14027&2,26085&1,62269&2,26085&1,62269&0,7992&5,304346149&1,323315349&7,019322677\\
1,14027&1,62269&2,26085&1&0,8612&1,55902&1,62269&2,26085&0,7992&5,304346149&1,323315349&7,019322677\\
1,14027&2,26085&1,62269&1&1,55902&0,8612&2,26085&1,62269&0,7992&5,304346149&1,323315349&7,019322677\\
1,62269&2,26085&1,14027&0,8612&1,55902&1&0,7992&1,62269&2,26085&5,304346149&1,323315349&7,019322677\\
2,26085&1,62269&1,14027&1,55902&0,8612&1&0,7992&2,26085&1,62269&5,304346149&1,323315349&7,019322677\\
0,8612&1,55902&1&1,62269&2,26085&1,14027&0,7992&1,62269&2,26085&5,304346149&1,323315349&7,019322677\\
1,55902&0,8612&1&2,26085&1,62269&1,14027&0,7992&2,26085&1,62269&5,304346149&1,323315349&7,019322677\\
1,62269&1,14027&2,26085&0,8612&1&1,55902&1,62269&0,7992&2,26085&5,304346149&1,323315349&7,019322677\\
2,26085&1,14027&1,62269&1,55902&1&0,8612&2,26085&0,7992&1,62269&5,304346149&1,323315349&7,019322677\\
0,8612&1&1,55902&1,62269&1,14027&2,26085&1,62269&0,7992&2,26085&5,304346149&1,323315349&7,019322677\\
1,55902&1&0,8612&2,26085&1,14027&1,62269&2,26085&0,7992&1,62269&5,304346149&1,323315349&7,019322677\\
\hline 1,07759&2,03039&2,03039&1,95073&2,82889&2,82889&1,25125&1,25125&1,42676&4,239231367&1,655800797&7,019322677\\
1,95073&2,82889&2,82889&1,07759&2,03039&2,03039&1,25125&1,25125&1,42676&4,239231367&1,655800797&7,019322677\\
2,03039&1,07759&2,03039&2,82889&1,95073&2,82889&1,25125&1,42676&1,25125&4,239231367&1,655800797&7,019322677\\
2,82889&1,95073&2,82889&2,03039&1,07759&2,03039&1,25125&1,42676&1,25125&4,239231367&1,655800797&7,019322677\\
2,82889&2,82889&1,95073&2,03039&2,03039&1,07759&1,42676&1,25125&1,25125&4,239231367&1,655800797&7,019322677\\
2,03039&2,03039&1,07759&2,82889&2,82889&1,95073&1,42676&1,25125&1,25125&4,239231367&1,655800797&7,019322677\\
\hline 1&0,49252&0,53073&0,7027&1,39327&0,61626&1,39327&0,61626&0,96077&8,60729454&0,815508595&7,019322677\\
0,49252&1&0,53073&1,39327&0,7027&0,61626&1,39327&0,96077&0,61626&8,60729454&0,815508595&7,019322677\\
0,53073&1&0,49252&0,61626&0,7027&1,39327&0,61626&0,96077&1,39327&8,60729454&0,815508595&7,019322677\\
0,61626&1,39327&0,7027&0,53073&0,49252&1&0,96077&0,61626&1,39327&8,60729454&0,815508595&7,019322677\\
0,7027&1,39327&0,61626&1&0,49252&0,53073&1,39327&0,61626&0,96077&8,60729454&0,815508595&7,019322677\\
1&0,53073&0,49252&0,7027&0,61626&1,39327&0,61626&1,39327&0,96077&8,60729454&0,815508595&7,019322677\\
1,39327&0,61626&0,7027&0,49252&0,53073&1&0,96077&1,39327&0,61626&8,60729454&0,815508595&7,019322677\\
0,49252&0,53073&1&1,39327&0,61626&0,7027&0,96077&1,39327&0,61626&8,60729454&0,815508595&7,019322677\\
0,53073&0,49252&1&0,61626&1,39327&0,7027&0,96077&0,61626&1,39327&8,60729454&0,815508595&7,019322677\\
0,61626&0,7027&1,39327&0,53073&1&0,49252&0,61626&0,96077&1,39327&8,60729454&0,815508595&7,019322677\\
0,7027&0,61626&1,39327&1&0,53073&0,49252&0,61626&1,39327&0,96077&8,60729454&0,815508595&7,019322677\\
1,39327&0,7027&0,61626&0,49252&1&0,53073&1,39327&0,96077&0,61626&8,60729454&0,815508595&7,019322677\\
\hline 0,71773&0,44231&0,50435&0,3535&0,68958&1&0,38092&0,71773&0,44231&11,99233214&0,585317568&7,019322677\\
0,44231&0,71773&0,50435&0,68958&0,3535&1&0,38092&0,44231&0,71773&11,99233214&0,585317568&7,019322677\\
0,50435&0,44231&0,71773&1&0,68958&0,3535&0,44231&0,71773&0,38092&11,99233214&0,585317568&7,019322677\\
0,3535&0,68958&1&0,71773&0,44231&0,50435&0,38092&0,71773&0,44231&11,99233214&0,585317568&7,019322677\\
0,3535&1&0,68958&0,71773&0,50435&0,44231&0,71773&0,38092&0,44231&11,99233214&0,585317568&7,019322677\\
0,50435&0,71773&0,44231&1&0,3535&0,68958&0,71773&0,44231&0,38092&11,99233214&0,585317568&7,019322677\\
0,68958&0,3535&1&0,44231&0,71773&0,50435&0,38092&0,44231&0,71773&11,99233214&0,585317568&7,019322677\\
0,68958&1&0,3535&0,44231&0,50435&0,71773&0,44231&0,38092&0,71773&11,99233214&0,585317568&7,019322677\\
0,71773&0,50435&0,44231&0,3535&1&0,68958&0,71773&0,38092&0,44231&11,99233214&0,585317568&7,019322677\\
1&0,3535&0,68958&0,50435&0,71773&0,44231&0,71773&0,44231&0,38092&11,99233214&0,585317568&7,019322677\\
1&0,68958&0,3535&0,50435&0,44231&0,71773&0,44231&0,71773&0,38092&11,99233214&0,585317568&7,019322677\\
0,44231&0,50435&0,71773&0,68958&1&0,3535&0,44231&0,38092&0,71773&11,99233214&0,585317568&7,019322677\\
\hline 0,71802&0,65848&0,43965&0,37889&1&0,7354&0,71802&0,42277&0,43965&11,41697016&0,614917816&7,020498352\\
0,37889&0,7354&1&0,71802&0,43965&0,65848&0,42277&0,71802&0,43965&11,41697016&0,614917816&7,020498352\\
0,37889&1&0,7354&0,71802&0,65848&0,43965&0,71802&0,42277&0,43965&11,41697016&0,614917816&7,020498352\\
0,43965&0,65848&0,71802&0,7354&1&0,37889&0,43965&0,42277&0,71802&11,41697016&0,614917816&7,020498352\\
0,43965&0,71802&0,65848&0,7354&0,37889&1&0,42277&0,43965&0,71802&11,41697016&0,614917816&7,020498352\\
0,65848&0,43965&0,71802&1&0,7354&0,37889&0,43965&0,71802&0,42277&11,41697016&0,614917816&7,020498352\\
0,65848&0,71802&0,43965&1&0,37889&0,7354&0,71802&0,43965&0,42277&11,41697016&0,614917816&7,020498352\\
0,7354&1&0,37889&0,43965&0,65848&0,71802&0,43965&0,42277&0,71802&11,41697016&0,614917816&7,020498352\\
1&0,37889&0,7354&0,65848&0,71802&0,43965&0,71802&0,43965&0,42277&11,41697016&0,614917816&7,020498352\\
1&0,7354&0,37889&0,65848&0,43965&0,71802&0,43965&0,71802&0,42277&11,41697016&0,614917816&7,020498352\\
0,71802&0,43965&0,65848&0,37889&0,7354&1&0,42277&0,71802&0,43965&11,41697016&0,614917816&7,020498352\\
0,7354&0,37889&1&0,43965&0,71802&0,65848&0,42277&0,43965&0,71802&11,41697016&0,614917816&7,020498352\\
\hline 0,66767&1,51864&1,09041&0,66767&1,51864&1,09041&1,11681&0,64203&0,57539&7,517906923&0,933836828&7,020498352\\
1,09041&0,66767&1,51864&1,09041&0,66767&1,51864&0,64203&0,57539&1,11681&7,517906923&0,933836828&7,020498352\\
1,09041&1,51864&0,66767&1,09041&1,51864&0,66767&0,57539&0,64203&1,11681&7,517906923&0,933836828&7,020498352\\
1,51864&0,66767&1,09041&1,51864&0,66767&1,09041&1,11681&0,57539&0,64203&7,517906923&0,933836828&7,020498352\\
1,51864&1,09041&0,66767&1,51864&1,09041&0,66767&0,57539&1,11681&0,64203&7,517906923&0,933836828&7,020498352\\
0,66767&1,09041&1,51864&0,66767&1,09041&1,51864&0,64203&1,11681&0,57539&7,517906923&0,933836828&7,020498352\\
\hline 0,89621&1,69838&1,69838&1,73949&1,03993&1,03993&2,36538&2,36538&1,55756&4,826722292&1,454506377&7,020498352\\
1,03993&1,03993&1,73949&1,69838&1,69838&0,89621&1,55756&2,36538&2,36538&4,826722292&1,454506377&7,020498352\\
1,03993&1,73949&1,03993&1,69838&0,89621&1,69838&2,36538&1,55756&2,36538&4,826722292&1,454506377&7,020498352\\
1,69838&0,89621&1,69838&1,03993&1,73949&1,03993&2,36538&1,55756&2,36538&4,826722292&1,454506377&7,020498352\\
1,69838&1,69838&0,89621&1,03993&1,03993&1,73949&1,55756&2,36538&2,36538&4,826722292&1,454506377&7,020498352\\
1,73949&1,03993&1,03993&0,89621&1,69838&1,69838&2,36538&2,36538&1,55756&4,826722292&1,454506377&7,020498352\\
\hline 1,89508&1,89508&1,11581&2,63932&2,63932&1,94095&1,73795&1,16037&1,16037&4,325741235&1,622958465&7,020498352\\
2,63932&2,63932&1,94095&1,89508&1,89508&1,11581&1,73795&1,16037&1,16037&4,325741235&1,622958465&7,020498352\\
1,94095&2,63932&2,63932&1,11581&1,89508&1,89508&1,16037&1,16037&1,73795&4,325741235&1,622958465&7,020498352\\
2,63932&1,94095&2,63932&1,89508&1,11581&1,89508&1,16037&1,73795&1,16037&4,325741235&1,622958465&7,020498352\\
1,11581&1,89508&1,89508&1,94095&2,63932&2,63932&1,16037&1,16037&1,73795&4,325741235&1,622958465&7,020498352\\
1,89508&1,11581&1,89508&2,63932&1,94095&2,63932&1,16037&1,73795&1,16037&4,325741235&1,622958465&7,020498352\\
\hline 0,52768&0,5888&1&1,39272&0,61231&0,91708&1,02421&1,39272&0,61231&8,197613566&0,856407575&7,020498352\\
0,5888&0,52768&1&0,61231&1,39272&0,91708&1,02421&0,61231&1,39272&8,197613566&0,856407575&7,020498352\\
0,5888&1&0,52768&0,61231&0,91708&1,39272&0,61231&1,02421&1,39272&8,197613566&0,856407575&7,020498352\\
0,61231&0,91708&1,39272&0,5888&1&0,52768&0,61231&1,02421&1,39272&8,197613566&0,856407575&7,020498352\\
0,61231&1,39272&0,91708&0,5888&0,52768&1&1,02421&0,61231&1,39272&8,197613566&0,856407575&7,020498352\\
1&0,52768&0,5888&0,91708&1,39272&0,61231&1,39272&0,61231&1,02421&8,197613566&0,856407575&7,020498352\\
0,52768&1&0,5888&1,39272&0,91708&0,61231&1,39272&1,02421&0,61231&8,197613566&0,856407575&7,020498352\\
0,91708&0,61231&1,39272&1&0,5888&0,52768&0,61231&1,39272&1,02421&8,197613566&0,856407575&7,020498352\\
0,91708&1,39272&0,61231&1&0,52768&0,5888&1,39272&0,61231&1,02421&8,197613566&0,856407575&7,020498352\\
1&0,5888&0,52768&0,91708&0,61231&1,39272&0,61231&1,39272&1,02421&8,197613566&0,856407575&7,020498352\\
1,39272&0,61231&0,91708&0,52768&0,5888&1&1,02421&1,39272&0,61231&8,197613566&0,856407575&7,020498352\\
1,39272&0,91708&0,61231&0,52768&1&0,5888&1,39272&1,02421&0,61231&8,197613566&0,856407575&7,020498352\\
\hline 1,6727&0,9616&1&2,27455&1,63317&1,49775&0,8618&2,27455&1,63317&5,019456934&1,398656955&7,020498352\\
1,49775&1,63317&2,27455&1&0,9616&1,6727&1,63317&2,27455&0,8618&5,019456934&1,398656955&7,020498352\\
1,63317&1,49775&2,27455&0,9616&1&1,6727&1,63317&0,8618&2,27455&5,019456934&1,398656955&7,020498352\\
1,6727&1&0,9616&2,27455&1,49775&1,63317&2,27455&0,8618&1,63317&5,019456934&1,398656955&7,020498352\\
2,27455&1,49775&1,63317&1,6727&1&0,9616&2,27455&0,8618&1,63317&5,019456934&1,398656955&7,020498352\\
2,27455&1,63317&1,49775&1,6727&0,9616&1&0,8618&2,27455&1,63317&5,019456934&1,398656955&7,020498352\\
0,9616&1&1,6727&1,63317&1,49775&2,27455&1,63317&0,8618&2,27455&5,019456934&1,398656955&7,020498352\\
0,9616&1,6727&1&1,63317&2,27455&1,49775&0,8618&1,63317&2,27455&5,019456934&1,398656955&7,020498352\\
1&0,9616&1,6727&1,49775&1,63317&2,27455&1,63317&2,27455&0,8618&5,019456934&1,398656955&7,020498352\\
1&1,6727&0,9616&1,49775&2,27455&1,63317&2,27455&1,63317&0,8618&5,019456934&1,398656955&7,020498352\\
1,49775&2,27455&1,63317&1&1,6727&0,9616&2,27455&1,63317&0,8618&5,019456934&1,398656955&7,020498352\\
1,63317&2,27455&1,49775&0,9616&1,6727&1&0,8618&1,63317&2,27455&5,019456934&1,398656955&7,020498352\\
\hline 1,35981&0,51521&1,35981&0,59784&0,57488&0,59784&0,97637&0,89541&0,97637&8,396031784&0,836168625&7,020498352\\
1,35981&1,35981&0,51521&0,59784&0,59784&0,57488&0,89541&0,97637&0,97637&8,396031784&0,836168625&7,020498352\\
0,51521&1,35981&1,35981&0,57488&0,59784&0,59784&0,97637&0,97637&0,89541&8,396031784&0,836168625&7,020498352\\
0,57488&0,59784&0,59784&0,51521&1,35981&1,35981&0,97637&0,97637&0,89541&8,396031784&0,836168625&7,020498352\\
0,59784&0,57488&0,59784&1,35981&0,51521&1,35981&0,97637&0,89541&0,97637&8,396031784&0,836168625&7,020498352\\
0,59784&0,59784&0,57488&1,35981&1,35981&0,51521&0,89541&0,97637&0,97637&8,396031784&0,836168625&7,020498352\\
\hline 0,33333&0,66667&1,33333&0,66667&0,33333&0,33333&0,33333&1&0,66667&12,000033&0,587241582&7,046918359\\
0,33333&1,33333&0,66667&0,66667&0,33333&0,33333&1&0,33333&0,66667&12,000033&0,587241582&7,046918359\\
0,66667&0,33333&0,33333&0,33333&1,33333&0,66667&1&0,33333&0,66667&12,000033&0,587241582&7,046918359\\
1,33333&0,33333&0,66667&0,33333&0,66667&0,33333&1&0,66667&0,33333&12,000033&0,587241582&7,046918359\\
1,33333&0,66667&0,33333&0,33333&0,33333&0,66667&0,66667&1&0,33333&12,000033&0,587241582&7,046918359\\
0,33333&0,33333&0,66667&0,66667&1,33333&0,33333&0,66667&0,33333&1&12,000033&0,587241582&7,046918359\\
0,33333&0,33333&0,66667&1,33333&0,66667&0,33333&0,66667&1&0,33333&12,000033&0,587241582&7,046918359\\
0,33333&0,66667&0,33333&0,66667&0,33333&1,33333&0,33333&0,66667&1&12,000033&0,587241582&7,046918359\\
0,33333&0,66667&0,33333&1,33333&0,33333&0,66667&1&0,66667&0,33333&12,000033&0,587241582&7,046918359\\
0,66667&0,33333&0,33333&0,33333&0,66667&1,33333&0,33333&1&0,66667&12,000033&0,587241582&7,046918359\\
0,66667&0,33333&1,33333&0,33333&0,66667&0,33333&0,33333&0,66667&1&12,000033&0,587241582&7,046918359\\
0,66667&1,33333&0,33333&0,33333&0,33333&0,66667&0,66667&0,33333&1&12,000033&0,587241582&7,046918359\\
\hline 0,5&0,25&0,75&0,5&0,75&0,25&0,25&0,25&0,5&16&0,440432397&7,046918359\\
0,5&0,75&0,25&0,5&0,25&0,75&0,25&0,25&0,5&16&0,440432397&7,046918359\\
0,75&0,25&0,5&0,25&0,75&0,5&0,5&0,25&0,25&16&0,440432397&7,046918359\\
0,75&0,5&0,25&0,25&0,5&0,75&0,25&0,5&0,25&16&0,440432397&7,046918359\\
\hline 0,5&1&0,5&1,5&2&0,5&0,5&1&1,5&8&0,880864795&7,046918359\\
1,5&0,5&2&0,5&0,5&1&1&0,5&1,5&8&0,880864795&7,046918359\\
1,5&2&0,5&0,5&1&0,5&0,5&1&1,5&8&0,880864795&7,046918359\\
\hline 3&1&2&4&2&3&2&1&1&4&1,76172959&7,046918359\\
3&2&1&4&3&2&1&2&1&4&1,76172959&7,046918359\\
\hline 0,25&0,5&0,75&0,75&0,5&0,25&0,25&0,5&0,25&16&0,440432397&7,046918359\\
0,25&0,75&0,5&0,75&0,25&0,5&0,5&0,25&0,25&16&0,440432397&7,046918359\\
\hline 1&2&3&2&3&4&1&2&1&4&1,76172959&7,046918359\\
1&3&2&2&4&3&2&1&1&4&1,76172959&7,046918359\\
2&1&3&3&2&4&1&1&2&4&1,76172959&7,046918359\\
2&3&1&3&4&2&1&1&2&4&1,76172959&7,046918359\\
2&3&4&1&2&3&1&2&1&4&1,76172959&7,046918359\\
2&4&3&1&3&2&2&1&1&4&1,76172959&7,046918359\\
3&2&4&2&1&3&1&1&2&4&1,76172959&7,046918359\\
3&4&2&2&3&1&1&1&2&4&1,76172959&7,046918359\\
\hline 0,5&0,5&1&1,5&0,5&2&1&0,5&1,5&8&0,880864795&7,046918359\\
1&0,5&0,5&2&1,5&0,5&0,5&1,5&1&8&0,880864795&7,046918359\\
1,5&0,5&0,5&0,5&0,5&1,5&1&2&1&8&0,880864795&7,046918359\\
1,5&0,5&0,5&0,5&1,5&0,5&2&1&1&8&0,880864795&7,046918359\\
2&1,5&0,5&1&0,5&0,5&0,5&1,5&1&8&0,880864795&7,046918359\\
\hline 4&2&3&3&1&2&2&1&1&4&1,76172959&7,046918359\\
4&3&2&3&2&1&1&2&1&4&1,76172959&7,046918359\\
\hline 0,5&1,5&2&0,5&0,5&1&1&1,5&0,5&8&0,880864795&7,046918359\\
0,5&2&1,5&0,5&1&0,5&1,5&1&0,5&8&0,880864795&7,046918359\\
2&0,5&1,5&1&0,5&0,5&1,5&0,5&1&8&0,880864795&7,046918359\\
\hline 2&1&2&3&2&1&1&4&3&4&1,76172959&7,046918359\\
2&2&1&1&3&2&4&3&1&4&1,76172959&7,046918359\\
2&2&1&3&1&2&4&1&3&4&1,76172959&7,046918359\\
3&1&2&2&2&1&4&1&3&4&1,76172959&7,046918359\\
3&2&1&2&1&2&1&4&3&4&1,76172959&7,046918359\\
\hline 0,5&0,5&1&0,5&1,5&2&1&1,5&0,5&8&0,880864795&7,046918359\\
0,5&0,5&1,5&0,5&1,5&0,5&1&1&2&8&0,880864795&7,046918359\\
0,5&0,5&1,5&1,5&0,5&0,5&1&2&1&8&0,880864795&7,046918359\\
0,5&1&0,5&0,5&2&1,5&1,5&1&0,5&8&0,880864795&7,046918359\\
0,5&1,5&0,5&0,5&0,5&1,5&1&1&2&8&0,880864795&7,046918359\\
0,5&1,5&0,5&1,5&0,5&0,5&2&1&1&8&0,880864795&7,046918359\\
1&0,5&0,5&2&0,5&1,5&1,5&0,5&1&8&0,880864795&7,046918359\\
\hline 1&2&2&2&1&3&3&1&4&4&1,76172959&7,046918359\\
1&2&2&2&3&1&1&3&4&4&1,76172959&7,046918359\\
1&2&3&2&1&2&3&4&1&4&1,76172959&7,046918359\\
1&3&2&2&2&1&4&3&1&4&1,76172959&7,046918359\\
2&1&2&1&2&3&3&4&1&4&1,76172959&7,046918359\\
2&1&3&1&2&2&3&1&4&4&1,76172959&7,046918359\\
2&3&1&1&2&2&1&3&4&4&1,76172959&7,046918359\\\hline\label{metric-f5}  \end{xtabular}}

\section*{Acknowledgments} \noindent {\textbf{Acknowledgments.}} R. M. Martins is supported by FAPESP-Brazil Project 2015/06903-8. Lino Grama is supported by FAPESP-Brazil Project 2014/17337-0. We thank the Department of Applied Mathematics of IMECC/Unicamp for providing us access to its computational facilities and Prof. Alberto Saa for useful discussions.

\end{document}